\documentclass[12pt,a4paper, twoside, reqno]{amsart} 

\usepackage[margin=2.5cm]{geometry}

\usepackage[T1]{fontenc} 
\usepackage[latin1]{inputenc} 
\usepackage[english]{babel} 

\usepackage{amsmath, amsthm, amsfonts, amssymb, amsbsy, bbm, varioref, enumerate, graphicx, ifpdf, url} 

\theoremstyle{definition}

\newtheorem*{defn*}{Definition}

\theoremstyle{remark}
\newtheorem*{rem}{Remark}
\theoremstyle{plain}
\newtheorem{thm}{Theorem}
\newtheorem*{thm*}{Theorem}
\newtheorem{lem}{Lemma}
\newtheorem{prop}{Proposition}

\newtheorem*{thm1a}{Theorem 1a}
\newtheorem*{thm1b}{Theorem 1b}
\newtheorem*{thm2}{Theorem 2}

\newcommand{\norm}[1]{\ensuremath{\left\Vert #1 \right\Vert}}

\newcommand{\infabs}[1]{\ensuremath{\left\vert #1 \right\vert}}

 \newtheorem{corollary}{Corollary}

\newcommand{\mc}{\mathcal}

  \newcommand{\Ha}{\mc{H}}

 \newcommand{\R}{\mathbb{R}}

 \newcommand{\N}{\mathbb{N}}
 
 \newcommand{\Z}{\mathbb{Z}}
  \newcommand{\Nc}{\mathcal{N}_{n,m}}
  \newcommand{\x}{\mathbf{x}}
  \newcommand{\y}{\mathbf{y}}

\title{A problem in non-linear Diophantine approximation}

\author{Stephen Harrap}

\address{S. Harrap: Durham University, Department of Mathematical Sciences, Science Laboratories, South Rd, Durham, DH1 3LE, United Kingdom} \email{s.g.harrap@durham.ac.uk}
\author{Mumtaz Hussain}\thanks{MH's research is supported by the Australian Research Council and La Trobe University's start-up grant.}

\address{M. Hussain, Department of Mathematics and Statistics, La Trobe University, POBox 199, Bendigo, Vic 3552, Australia.}\email{m.hussain@latrobe.edu.au}

\author{Simon Kristensen}

\address{S. Kristensen: Department of Mathematics, Aarhus University, Ny Munkegade 118, 8000 Aarhus C, Denmark}
\email{sik@math.au.dk}

\thanks{SK's research is supported by the Danish Natural Science Research Council.}

\begin{document}
	
	\begin{abstract} In this paper we obtain the Lebesgue and Hausdorff measure results for the set of vectors satisfying infinitely many fully non-linear Diophantine inequalities.  The set is associated with a class of linear inhomogeneous partial differential equations whose solubility depends on a certain Diophantine condition. The failure of the Diophantine condition guarantees the existence of a smooth solution. 
	\end{abstract}
	
	\maketitle
	
	
	\section{Introduction and Statements of results}
	\label{sec:intro}
	
	Metric Diophantine approximation of a single linear form is in a first instance concerned with the (Lebesgue and Hausdorff) measure of the set of vectors $(x_1, \dots, x_k) \in \mathbb{R}^k$ for which there are infinitely many integral vectors $(q_1, \dots, q_k, p) \in \mathbb{Z}^{k+1}$ satisfying the inequality
	\begin{equation}
		\label{eq:1}
		\vert q_1 x_1 + \dots +q_k x_k - p \vert < \psi(H(\mathbf{q})).
	\end{equation}
	 Here, $\psi : \mathbb{N} \rightarrow \mathbb{R}_{>0}$ denotes a monotonic arithmetic function decreasing to zero and $H(\mathbf{q})$ denotes the naive height of the vector $\mathbf{q}$, i.e. $H(\mathbf{q}) = \max\{\vert q_1 \vert, \dots, \vert q_k \vert \}$. This set is well studied. Its Lebesgue measure is determined by the famous zero--one law of Groshev \cite{Gros} and its Hausdorff measure was calculated by Dickinson \& Velani \cite{DV97}. We refer to \cite{BDV06, BV09, BV10_IMRN} for refined modern results in this direction. For obvious reasons $\psi$ is often referred to as an \textit{approximating function}.
	
	 When one places restrictions on the set of integral vectors for which \eqref{eq:1} is required to have infinitely many solutions, the problem is less studied. Constraints on the `denominator' terms $\mathbf{q}$ alone are relatively easily dealt with. For instance, appealing to a result of Schmidt \cite{schmidt}, one can calculate the Lebesgue measure of the above set when coordinates of $\mathbf{q}$ are restricted to values of prescribed polynomials. By appealing to a result of Rynne \cite{R92_2} one may determine the Hausdorff dimension for far more general restrictions. However, both of these works require the `numerator' term $p$ to be restriction free.
	
	Introducing constraints on both the denominator terms and the numerator terms can prove very problematic as many known methods break down. In a series of papers, Harman \cite[Chapter 6 and the references therein]{Har_98} explored approximation of a real numbers by rationals whose numerators and denominators came from pre-specified sets, and in doing so gave definitive answers in the one-dimensional `$k=1$' case.  
	
	For $k > 1$, a linear condition on the numerator term $p$ is easily dealt with by a change of variables and the case that $p$ is restricted to being equal to zero is well studied~\cite{DH13, FHKL15, HK13_2, HL13}. However, to our knowledge the only complete metrical results for non-linear constraints on both numerator and denominator terms known to date is those of \cite{Jones} and~\cite{BDoKL08}, in which the numerators and denominators are all assumed to be primes and perfect squares respectively. The results of \cite{Jones} were later extended (\cite{HJ}, and very recently~\cite{BG}) to intersections with lines and planar curves.  	In this paper, we extend ~\cite{BDoKL08} to a more general setup in which numerators and denominators are required to be fixed, possibly unlike, powers of integers.

	Determining the measure-theoretic properties of the sets described above is not just an exercise in measure and number theory. Indeed, such Diophantine inequalities can be encountered in the study of solutions to certain PDEs and has attracted sustained interest (e.g., see \cite{DGY, DPRY, DV, krist_03, Pet1, Pet2}). It can often be shown that the exceptional set of points where these inequalities fail to hold is small, typically of zero Lebesgue measure. Outside of these sets, the differential equations under consideration are guaranteed to have solutions and so a more acute understanding of the `size' of the exceptional sets becomes a question of real interest.
	
	For example, let $f: \R^3 \rightarrow \R$ be periodic in each of its variables $x_1, x_2$ and $t$, with periods $\alpha$, $\beta$ and $\gamma$ respectively. Assume also that $f$ is a smooth function of each variable. The inhomogeneous wave equation studied in \cite{BDoKL08} is given by the PDE
	\begin{equation}
		\label{eqn:wave}
		\frac{\partial^2 u(\mathbf{x}, t)}{\partial t^2} \, - \, \frac{\partial^2 u(\mathbf{x}, t)}{\partial x_1^2}\, - \, \frac{\partial^2 u(\mathbf{x}, t)}{\partial x_2^2} \, = \, f(\mathbf{x}, t), \quad  \mathbf{x}=(x_1, x_2) \in \R^2, \, t \in \R,
	\end{equation}
	where $u$ is a smooth, periodic solution with the same periods as $f$. 
	It is well known that the smoothness conditions on $f$ are equivalent to the property that it has a Fourier series expansion of the form
	\begin{equation}\label{eq:fourier_series}
		f(\mathbf{x}, t) \, = \, \sum_{(a, b, c)\in \Z^3} \: f_{a, b, c} \, \exp\left( 2\pi i \left[ \frac{a}{\alpha}x_1 + \frac{b}{\beta}x_2 + \frac{c}{\gamma}t \right]\right),
	\end{equation}
	in which the coefficients $f_{a, b, c}$ decay suitably quickly. Any smooth solution  $u$ to (\ref{eqn:wave}) must satisfy a similar Fourier expansion. Upon comparing coefficients one can deduce that a sufficient condition for $u$ to be smooth is that there is no real number $\tau > 1$ such that the Diophantine inequality
	\begin{equation}
		\label{eqn:wavey}
		\infabs{a^2 \frac{\gamma^2}{\alpha^2} + b^2 \frac{\gamma^2}{\beta^2} - c^2} \: < \: \max\left\{\infabs{a}, \infabs{b}\right\}^{-\tau}
	\end{equation}
	holds for infinitely many $(a, b, c)\in \Z^3$ with $(a, b) \neq (0, 0)$. In other words, a solution to~\eqref{eqn:wave} is guaranteed to exist if the ratio of certain functions of the periods of the functions in the PDE are not a set given by \eqref{eq:1} under the condition that $n=2$ and the condition that $q_1, q_2$ and $p$ are all required to be perfect squares.
	
	Of course, one could consider other PDEs, where the wave operator on the left hand side of~\eqref{eqn:wave} is replaced by a differential operator of the form
	\begin{equation*}
	\frac{\partial^p}{\partial t^p} \, - \, \frac{\partial^n}{\partial x_1^n}\, - \, \frac{\partial^m}{\partial x_2^m}.
	\end{equation*}
	This would lead to a more general Diophantine obstruction.
	It is therefore natural to investigate more general inequalities of the form~(\ref{eqn:wavey}) from a metrical point of view. 
	
	In order for our Diophantine results below to be applicable to this situation, we need to impose additional regularity on the inhomogeneous term. Concretely, we can replace the assumption of smoothness by the stronger assumption of having an extension as a holomorphic function of several complex variables $(x_1, x_2, t)$ to a cartesian product of annuli  $D$ containing the torus on which $f$ is defined (recall that periodicity in the three variables is the same as saying that $f$ is defined on a torus). This would imply that $f$ has a Laurent series expansion in $D$ with only finitely many negative exponents in either variable. Restricting this expansion to the torus on which $f$ is defined, we arrive at a Fourier series as before in \eqref{eq:fourier_series}, but now the coefficients $f_{a,b,c}$ are identically zero whenever one of $a$, $b$ or~$c$ is small enough.
	
	Recall from \cite{BDoKL08} that the Diophantine condition of \eqref{eqn:wavey} is only relevant whenever $f_{a,b,c}~\neq~0$, and its predominant purpose is  to ensure that the series expansion for a solution $u$ will converge. This is also the case with the more general partial differential operators considered in this paper. As such, in this new case we need only worry about $a$, $b$ and $c$ being positive when considering the Diophantine condition. As a result, if the ratios of periods are such that for no $\tau > 1$ the analogue of \eqref{eqn:wavey} with the new exponents does not have infinitely many solutions,  we can find a solution to the PDE, and it will satisfy the same regularity as~$f$ did; namely it will have holomorphic continuation to a product of annuli. 
	
	We suspect that the extra regularity assumption is not strictly needed for our results to hold, but the technicalities of the number theoretical arguments would increase in a spectacular way, and we have chosen to make this extra regularity assumption for clarity. Of course, if both $m$ and $n$ are even, our results are applicable with just the assumption of smoothness. Indeed, in this case the condition does not distinguish between positive and negative values of $a$ and $b$, and we may return to the original situation.
	
	One technical condition of our arithmetical result remains, which has an impact on the classes of partial differential operators we can deal with; when $n \neq m$ we require either that $p=1$, or that $n$ and $m$ must have a common non-trivial divisor. This is a real obstruction to the applicability of our result, and one which cannot be overcome by imposing extra assumptions of regularity.

	With the motivations and limitations described above, we now define the sets whose size we are going to calculate.	For any triple $(n, m, p)\in \N^3$ and any approximating function~$\psi$ define $W_{n, m}^p(\psi)$ to be the set of vectors $\mathbf{x}=(x_1, x_2)\in [0,1)^2$ for which the inequality
	$$
	\infabs{a^nx_1 + b^mx_2-c^p} \, < \, \psi(h_{a,b})
	$$
	holds for infinitely many $(a,b,c) \in \N^2 \times \Z_{\geq 0}$. 
	Here, we have assigned a natural height $h_{a,b}:=\max(a^n, b^m)$ to each pair $(a,b)$ of positive integers.  Associated to each approximating function $\psi$, there is a quantity $\lambda_\psi \in \left[0,\infty\right]$ as defined by Dodson \cite {Dodson}, given by 
	$$
	\lambda_\psi : \, = \, \liminf_{r \rightarrow \infty} \frac{- \log\psi(2^r)}{r \log 2};
	$$
	\textit{the lower order of $1/\psi$ at infinity}. This will also be of use to us here.
	
	We provide a Groshev-like criterion for the size of $W_{n, m}^p(\psi)$ in terms of the convergence and divergence of a certain sum. In particular, we prove the following result, in which $\infabs{ \, \cdot \,}$ denotes normalized Lebesgue measure on $[0,1]^2$. 
	The nature of the sets involved depends intricately on the values of $p$, $n$ and $m$. For clarity we split our statement into two parts, $n=m$ and $n\neq m$ respectively, for in the former case the statement is much cleaner.
	
	\begin{thm1a}\label{thm:main} For every approximating function $\psi$ we have that
		$$
		\infabs{ W_{n, n}^p(\psi)} \quad = \quad
		\begin{cases}
		0, & \displaystyle\sum_{q \in \N} \psi(q)/q^{2-\frac 2n-\frac1p} \:  < \, \infty. \\
		& \\
		1, & \displaystyle\sum_{q \in \N} \psi(q)/q^{2-\frac 2n-\frac1p} \:  = \, \infty.
		\end{cases}
		$$
	\end{thm1a}
 For some choices of natural numbers $n, m$ and $p$ the set $W_{n, m}^p(\psi)$ will always be a Lebesgue null set~(see~\S\ref{sec:restrictionsp}). Indeed, when  $n \neq m$ we will assume
\begin{equation}
\label{eqn:pbound}
\frac 1n + \frac 1m + \frac 1p \: > \: 1.
\end{equation}
For, when (\ref{eqn:pbound}) does not hold the set $W_{n, m}^p(\psi)$ is always Lebesgue null in the case of strict inequality, and null for any $\psi$ with $\lambda_\psi>0$ in the case of equality. 
	
	\begin{thm1b}\label{thm:mainb} Fix $n \neq m$ satisfying (\ref{eqn:pbound}). If  either $p=1$ or $\gcd(n,m)\geq 2$, then  
		for every approximating function $\psi$ we have that
		$$
		\infabs{ W_{n, m}^p(\psi)} \quad = \quad
		\begin{cases}
		0, & \displaystyle\sum_{\substack{(a, b)\in \, \N^2}}  \: \frac{\psi(h_{a,b})}{h_{a,b}^{1-1/p}} \:  < \, \infty. \\
		& \\
		1, & \displaystyle\sum_{\substack{(a, b)\in \, \N^2 }}  \: \frac{\psi(h_{a,b})}{h_{a,b}^{1-1/p}} \:  = \, \infty.
		\end{cases}
		$$
	\end{thm1b}
		When $n=m=p=1$ our result coincides with a case of the famous theorem of Groshev~\cite{Gros} and when $n=m=p=2$ the main result of \cite{BDoKL08}. To see that our result contains the latter for their approximating function~$\phi$, note that our height function is defined slightly differently so that $\phi(q) = \psi(q^2)$. Thus, our $\sum_{q \in \N} \psi(q)/q^{1/2}$ is equivalent to the sum $\sum_{q \in \N} \phi(q)$ of \cite{BDoKL08}. 
	\begin{rem}
		Note that when $\gcd(n,m)\geq2$ equation \eqref{eqn:pbound} provides a natural upper bound for $p$ given by $nm/(nm-n-m)$. This would not necessarily be true in  the case of coprimeness.
	\end{rem}

	
	In their own right, Theorems 1a and 1b give no further information on how to distinguish between sets it has determined to have Lebesgue measure zero. 
	Intuitively, the size of $W_{n, m}^p(\psi)$ should still decrease as the rate of approximation governed by the approximating function $\psi$ increases (assuming the sets concerned are non-empty). Hausdorff measure and dimension are the appropriate tools to distinguish amongst such exceptional sets. To this end,  we now provide a general criterion for the size of the set  $W_{n, m}^p(\psi)$ in terms of these tools. Throughout, by a \textit{dimension function} we mean an increasing function $f:\R\to\R$ such that  $f(r)\to 0$ as $r\to 0$. As usual, we denote $f$-dimensional Hausdorff measure  by $\Ha^f$ and Hausdorff dimension by $\dim_H$. For  precise definitions see \S\ref{HM}.
	
	\begin{thm2}\label{thm:HM}
		
		Let $\psi$ be an approximating function and let $p$, $n$ and $m$ be as in Theorems~1a or~1b. 		Let $f$ be a dimension function such that $r^{-2}f(r)$ is monotonic and for notational convenience let $g:r\to r^{-1}f(r)$ be another dimension function. Then, 
		$$
		\Ha ^f\left(W_{n, m}^p(\psi)\right) \quad = \quad
		\begin{cases}
		0, & \displaystyle\sum_{\substack{(a, b)\in \, \N^2}}  \: g\left(\frac{\psi(h_{a,b})}{h_{a,b}}\right)h_{a, b}^{1/p} \:  < \, \infty. \\
		& \\
		\Ha^f\left([0, 1)^2\right), & \displaystyle\sum_{\substack{(a, b)\in \, \N^2 }}  \: g\left(\frac{\psi(h_{a,b})}{h_{a,b}}\right)h_{a, b}^{1/p} \:  = \, \infty.
		\end{cases}
		$$
	\end{thm2}
	
	Armed with this theorem, we are able to extend a result of Dodson~\cite{Dodson} upon setting $f:r\to r^s$ for some $s>0$. Dodson's result \cite{Dodson} corresponds to the case $n=m=p=1$ in the below statement, and $n=m=p=2$ to Corollory~3.3 from \cite{BDoKL08}.
	
		\begin{corollary} \label{cor:HDorder}
			Let $\psi$ be an approximating function and assume that either $p=1$, $n=m=1$ or $\gcd(n,m)\geq 2$. 
			If $0 < \lambda_\psi < \infty$, then
			\[\dim_H\left(W_{n, m}^p \left( \psi\right)\right)=1+\min\left\{1, \: \frac{\frac1n+\frac1m+\frac1p}{\lambda_{\psi}+1}\right\}.\]
		\end{corollary}
	
	
	When $\psi:r\to r^{-\tau}$ for some $\tau>1$, the following Hausdorff dimension statement can readily be obtained. This is in perfect analogue with Corollory~3.4 from \cite{BDoKL08}. 
	
	\begin{corollary} \label{cor:HD}
	Assume that either $p=1$, $n=m=1$ or $\gcd(n,m)\geq 2$.		Let $\tau>0$, then
		\[\dim_H\left(W_{n, m}^p \left( \psi:r\to r^{-\tau}\right)\right)=1+\min\left\{1, \: \frac{\frac1n+\frac1m+\frac1p}{\tau+1}\right\}.\]
	\end{corollary}
	In fact, Theorem~2 reveals more than just the Hausdorff dimension of the sets concerned. It also implies, for example, that  $\mc H^s\left(W_{n, m}^p \left( \psi:r\to r^{-\tau}\right)\right)=\infty$ at the critical exponent $s=\dim_H\left(W_{n, m}^p \left( \psi:r\to r^{-\tau}\right)\right)$.

	
	\section{Proof of Theorems 1a \& 1b}
	\label{sec:ProofOfTheoremRefThmMain}

	We will use the following notation throughout the proof. Fix any two natural numbers $a$ and $b$. Then, for every $c \in \Z_{\geq 0}$ let
	$$
	\ell_{a,b}(c): \, =\, \left\{(x, y) \in [0,1)^2: \, \infabs{a^nx+b^my-c^p} < \psi(h_{a,b})\right\}
	$$
	and in turn let
	$$
	\ell_{a,b} \, = \, \bigcup_{c \, \in \, \Z_{\geq 0}}\ell_{a,b}(c).
	$$
	Each set $\ell_{a,b}(c)$ is simply a `strip' in $[0,1)^2$ consisting of a segment of a certain~neighbourhood of the line $y=(-a^n/b^m)x+c^p/b^m$. The set $\ell_{a,b}$ is the disjoint union as $c$ runs over $\Z_{\geq 0}$ of all such strips which are non-empty. Moreover, this notation gives us a very convenient way of expressing the set $W_{n,m}^p(\psi)$. Indeed, we have
	$$
	W_{n,m}^p(\psi) \, = \, \left\{(x, y) \in [0,1)^2: \, (x,y) \in \ell_{a,b} \text{ for infinitely many pairs } (a,b) \in \N^2 \right\}.
	$$
	
	As is now commonplace in number theory we will often appeal to Vinogradov notation rather than `big O' notation to allow for neatness of exposition. For the unfamiliar reader we mean by $f \ll g$ that $f(t)=\mathcal{O}(g(t))$ as $t \rightarrow \infty$ and by $f \asymp g$ that both $f \ll g$ and $f \gg g$. We consider $n,m$ and $p$ to be fixed throughout the proof and any implied constants may depend on these integers alone.

	We begin the proof of Theorem \ref{thm:main} by dealing with the case when the volume sum converges. As is usual for results of this type, it takes the form of a simple covering argument.

	\subsection{Proof of the convergence part}
	\label{sec:conv}
	
	For any fixed $c$ it is easily verified that the strip $\ell_{a,b}(c)$ has measure at most $2\sqrt{2}\psi(h_{a,b})/\sqrt{a^{2n}+b^{2m}} \ll \psi(h_{a,b})/h_{a,b}$ and a simple calculation yields that there are at most $2h_{a,b}^{1/p}+1$ non-empty strips in the union $\ell_{a,b}$. Hence,
	$$
	\infabs{\ell_{a, b}} \, \ll \frac{\psi(h_{a, b})}{h_{a, b}^{1-1/p}}. 
	$$
	Furthermore, the function $h_{a,b}: \N^2 \rightarrow \N$ only takes values which are $n$-th or $m$-th powers, and so for any natural number $h$ not of this form we have  
	$$
	\displaystyle\bigcup_{\substack{(a, b)\in \, \N^2 \\ h_{a,b} =h}}\ell_{a,b} \: = \: \emptyset.
	$$
	Therefore, 
	\begin{eqnarray*}
		\sum_{h=1}^\infty \: \infabs{\displaystyle\bigcup_{(a, b)\in \, \N^2, \: h_{a,b} =h} \, 
			\ell_{a,b}} 
		& = & \sum_{g_1=1}^\infty \: \infabs{\displaystyle\bigcup_{b^m \, \leq \, g_1^n} \, 
			\ell_{g_1,b}} \: + \: \sum_{g_2=1}^\infty \: \infabs{\displaystyle\bigcup_{a^n \, < \, g_2^m} \, 
			\ell_{a,g_2}} \\
		& \ll &  \sum_{g_1=1}^\infty \:\displaystyle\sum_{\substack{b\in \, \N \\ b^m \leq g_1^n}} \, \frac{\psi(g_1^n)}{g_1^{n(1-1/p)}} \: + \: 
		\sum_{g_2=1}^\infty \: \displaystyle\sum_{\substack{a\in \, \N \\ a^n < g_2^m}} \, \frac{\psi(g_2^m)}{g_2^{m(1-1/p)}} \\
		& = & \displaystyle\sum_{\substack{(a,b)\in \, \N^2 \\ h_{a,b}=a^n}} \, \frac{\psi(h_{a,b})}{h_{a,b}^{1-1/p}} \: + \: 
		\displaystyle\sum_{\substack{(a,b)\in \, \N^2 \\ h_{a,b}=b^m > a^n}} \, \frac{\psi(h_{a,b})}{h_{a,b}^{1-1/p}}  \\
		&=& \displaystyle\sum_{(a,b)\in \, \N^2} \, \frac{\psi(h_{a,b})}{h_{a,b}^{1-1/p}} \quad < \quad \infty. \
	\end{eqnarray*}
	Since 
	$$
	W_{n,m}^p(\psi) \: = \: \bigcap_{g=1}^\infty \bigcup_{h=g}^\infty \left(  \bigcup_{(a, b)\in \, \N^2, \: h_{a,b} =h}	\ell_{a,b} \right),
	$$
	it follows from the `convergence part' of the famous Borel-Cantelli lemma in probability theory that $W_{n, m}^p(\psi)$ is a null set as required. Note that this holds for \textit{any} choice of $n$, $m$ and $p$, not only those satisfying the conditions of Theorem 1b.

	\subsection{Preliminaries for the proof of the divergence part}
	
	Proving the validity of Theorems 1a and 1b for the case when the volume sum
	\begin{equation}\label{eqn:volsum}
		\displaystyle\sum_{\substack{(a, b)\in \, \N^2}} \: \frac{\psi(h_{a,b})}{h_{a,b}^{1-1/p}}
	\end{equation}
	diverges constitutes the main difficulty in the proof as a whole and will require some very delicate calculations. As such, before we proceed we are first required to provide some auxiliary lemmas and outline some important general observations.

\subsubsection{An important observation}
\label{sec:restrictionsp}

Since $\psi(r)\rightarrow 0$ as $r \rightarrow \infty$ we may assume that $\psi(r) < 1$ for $r$ sufficiently large. 
Observe that we cannot have
\begin{equation} \label{badbound}
	\frac 1n + \frac 1m + \frac 1p \: < \: 1,
\end{equation}
for otherwise it would follow that
$$
\min\left\{n(1-1/p-1/m),\: m(1-1/p-1/n)\right\} \: > \: 1,
$$
and for any sufficiently large $H \in \N$ we would have
\begin{eqnarray*}
\displaystyle\sum_{\substack{(a,b)\in \, \N^2: \\ h_{a,b} \geq H}} \, \frac{\psi(h_{a,b})}{h_{a,b}^{1-1/p}} & = & \sum_{g_1\geq \left\lfloor H^{1/n} \right\rfloor} \:\displaystyle\sum_{\substack{b\in \, \N \\ b^m \leq g_1^n}} \, \frac{\psi(g_1^n)}{g_1^{n(1-1/p)}} \: + \: 
 \sum_{g_2 \geq\left\lfloor H^{1/m} \right\rfloor} \: \displaystyle\sum_{\substack{a\in \, \N \\ a^n < g_2^m}} \, \frac{\psi(g_2^m)}{g_2^{m(1-1/p)}}\\
 & \ll & 
  \sum_{g_1\geq \left\lfloor H^{1/n} \right\rfloor} \: \frac{ g_1^{n/m} }{g_1^{n(1-1/p)}} \: + \: 
  \sum_{g_2\geq \left\lfloor H^{1/m} \right\rfloor} \: \frac{g_2^{m/n} }{g_2^{m(1-1/p)}} \quad < \: \infty,\
\end{eqnarray*}
a contradiction. Moreover, in view of \S\ref{sec:conv} this observation implies that the set $W_{n, m}^p(\psi)$ has measure zero for every approximating function $\psi$ whenever \eqref{badbound} holds.
One can readily verify this is also the case when we have equality in (\ref{badbound}) and $\lambda_\psi \neq 0$. We therefore proceed on the assumption that (\ref{eqn:pbound}) holds when  $n \neq m$, as in Theorem~1b.

\subsubsection{Restrictions on the integers $(a,b)$}
\label{sec:restrictions}

It will be imperative to our proof that we exclude a certain class of integer pairs $(a,b) \in \N^2$ from our calculations. Firstly, we will exclude those pairs for which $a$ and $b$ are not coprime. The reason for this is that it will guarantee that the strip $\ell_{a,b}$ is not parallel to any other strip we might consider. In view of the method outlined in \S\ref{sec:lemmas} this will be very important, as otherwise the intersection of any two such strips may be very large. 

Secondly, we will assume that the resonant lines at the centre of all our strips lie in some `cone'; that is, the angle of incline of each strip $\ell_{a,b}$ is neither too steep or too flat. The reason for this assumption will become apparent as our proof progresses. 
To be precise, for the rest of the proof we will work exclusively with pairs $(a,b) \in \Nc$, where $\Nc \subset \N^2$ denotes the set of pairs $(a,b) \in \N^2$ satisfying the conditions
\begin{equation} \label{eqn:restrictions}
	\gcd(a, b)=1, \quad\quad \frac{1}{2} \, < \, \frac{a^n}{b^m} \, < \, 2.
\end{equation}

Before we proceed we must first ensure that this thinning out of the sequence of sets $\ell_{a,b}$ does not effect the implication of our proof. To see that it does not, notice that the set 
$$
V_{n,m}^p(\psi) \, = \, \left\{(x, y) \in [0,1)^2: \, (x,y) \in \ell_{a,b} \text{ for infinitely many pairs } (a,b) \in \Nc \right\}
$$
is a subset of $W_{n,m}^p(\psi)$. Therefore, if we can prove that $V_{n,m}^p(\psi)$ has full measure then it will readily follow that $W_{n,m}^p(\psi)$ also enjoys this property. Our proof of Theorems 1a and 1b would then be complete modulo the following proposition, which demonstrates that the sequence of strips $\ell_{a,b}$ for $(a,b) \in \Nc$ is `rich' enough to entirely determine whether the volume sum (\ref{eqn:volsum}) diverges.

\begin{prop}
\label{prop:restrict}
	For any approximating function $\psi$ and any triple $(n,m,p) \in \N^3$ we have
	$$
	\sum_{\substack{(a, b)\in \, \N^2 }}  \: \frac{\psi(h_{a, b})}{h_{a, b}^{1-1/p}} \: = \: \infty \quad  \quad\Longleftrightarrow \quad \quad
	\displaystyle\sum_{(a, b)\in \, \Nc} \, \frac{\psi(h_{a, b})}{h_{a, b}^{1-1/p}}  \: = \: \infty.$$
\end{prop}
To prove this proposition we will require the following two lemmas (which may be deduced from standard arithmetic identities, e.g., see  \cite{Apo76}). Throughout, $\varphi$ will denote Euler's totient function and $\sigma$ will denote the divisor function.
\begin{lem}\label{lem:rest1}
	Choose a fixed natural number $t$, and then for any other natural number $Q$ denote by $\gamma_t(Q)$ the cardinality of the set $\left\{q \leq Q: \, \gcd(t,q)=1 \right\}.$ Then, 
	$$
		\gamma_t(Q) = \frac{\varphi(t)}{t}Q \, + \epsilon_t(Q),
	$$	
	where $\epsilon_t(Q):\N \rightarrow \R$ is an error function satisfying  $|\epsilon_t(Q)| \leq \sigma(t)$ for every $Q\in\N$. 
\end{lem}
	\begin{lem}\label{lem:rest2}
	For any fixed real number $z >0$, we have 
	$$
		\sum_{q=1}^Q q^{z-1}\varphi(q) \, = \, \frac{6}{\pi^2(z+1)}\, Q^{z+1} \, + \mathcal{O}(Q^z\log Q)
	$$	
	as $Q \rightarrow \infty$. 
\end{lem}
It suffices to show that	$$
	\sum_{\substack{(a, b)\in \, \N^2 }}  \: \frac{\psi(h_{a, b})}{h_{a, b}^{1-1/p}} \: = \: \infty \quad  \quad\Longrightarrow \quad \quad
	\displaystyle\sum_{(a, b)\in \, \Nc} \, \frac{\psi(h_{a, b})}{h_{a, b}^{1-1/p}}  \: = \: \infty$$
	since the complementary implication is obvious. To do this we will show that if the latter sum converges then so does the former. 
	In order to proceed we first partition the set $\Nc$ by separating the heights $h_{a,b}$ into dyadic blocks. For any $t \in \Z_{\geq 0}$ let 
$$
A_t: = \left\{ (a,b) \in \Nc: \: a^n>b^m,\: 2^{t+1} > a \geq 2^t \right\},
$$
and let
$$
B_t: = \left\{ (a,b) \in \Nc: \: b^m>a^n, \: 2^{t+1} > b \geq 2^t \right\}.
$$
Note that by definition $\Nc$ is precisely the disjoint union $\cup_{t \in \Z_{\geq 0}}(A_t \cup B_t)$. We will denote by $\alpha_t$ the cardinality of the set $A_t$ and by $\beta_t$ the cardinality of the set $B_t$. We may re-express the set $A_t$ in the following way:
$$
A_t = \left\{ (a,b) \in \N^2: \, \gcd(a,b)=1, \, \left\lfloor a^\frac{n}{m}\right\rfloor>b>\left\lfloor 2^{-\frac{1}{m}} \,a^\frac{n}{m}  \right\rfloor,\: 2^{t+1} > a \geq 2^t \right\}.
$$
Recall the following well known expression (e.g. \cite{Apo76} Theorem $3.3$) describing the behaviour of the divisor summatory function:
$$
\sum_{q=1}^Q \: \sigma(q) \, = \, Q \,\log Q \, + \mathcal{O}(Q)
	$$	
	as $Q \rightarrow \infty$. With reference to the notation of Lemma \ref{lem:rest1}, an immediate consequence of the above statement is that for any sequence of natural numbers $\left\{N_q\right\}_{q=1}^{\infty}$ we have
\begin{equation} \label{eqn:error}
\infabs{\sum_{q=2^t+1}^{2^{t+1}} \epsilon_q(N_q)} \, \leq \, \sum_{q=2^t+1}^{2^{t+1}} \infabs{\epsilon_q( N_q)} \, \leq \, \sum_{q=2^t+1}^{2^{t+1}} \sigma(q) \, \ll \, t2^t, 
\end{equation}
for the error function appearing in Lemma \ref{lem:rest1}.
Finally, we recall (e.g. \cite{Apo76} Chapter $3$, Ex. $5(b)$) the property that
$$
 \sum_{d=1}^Q \frac{\varphi(d)}{d} \, = \, \frac{Q}{\zeta(2)} \, + \, \mathcal{O}(\log Q),
	$$
	as $Q \rightarrow \infty$.
Hence, on applying Lemma \ref{lem:rest1} to the set $A_t$ and utilising (\ref{eqn:error}) we conclude that
\begin{eqnarray*}
	\alpha_t & \asymp & \sum_{a=2^t+1}^{2^{t+1}} \left(\gamma_a\left(\left\lfloor 
	a^\frac{n}{m}\right\rfloor-1\right) - \gamma_a\left(\left\lfloor 	2^{-\frac{1}{m}}
	\,a^\frac{n}{m}  \right\rfloor \right)\right)\\
	& = & \sum_{a=2^t+1}^{2^{t+1}}\left(\frac{\varphi(a)}{a} \left(\left\lfloor 
	a^\frac{n}{m}\right\rfloor-1- \left\lfloor 	2^{-\frac{1}{m}}
	\,a^\frac{n}{m}  \right\rfloor \right) + \epsilon_a\left(\left\lfloor 
	a^\frac{n}{m}\right\rfloor-1\right) - \epsilon_a\left(\left\lfloor 	2^{-\frac{1}{m}}
	\,a^\frac{n}{m}  \right\rfloor \right)\right) \\
	& = & \left(1-2^{-\frac{1}{m}}\right) \sum_{a=2^t+1}^{2^{t+1}}a^{\frac{n}{m}-1}\varphi(a) \: + \: \mathcal{O}(t2^t)\	
\end{eqnarray*}  
as $t \rightarrow \infty$. By Lemma \ref{lem:rest2} it follows that
\begin{eqnarray*}
	\alpha_t & = &  \left(1-2^{-\frac{1}{m}}\right)  \frac{2^{(\frac{n}{m}+1)(t+1)}-2^{(\frac{n}{m}+1)t}}{(\frac{n}{m}+1)\zeta(2)} \: + \: \mathcal{O}\left(\max\left(t2^t, \, t2^{\frac{n}{m}t}\right) \right)   \\
	& = & \frac{\left(1-2^{-\frac{1}{m}}\right)\left( 2^{(\frac{n}{m}+1)}-1\right)}{(\frac{n}{m}+1)\zeta(2)}
	\, 2^{(\frac{n}{m}+1)t}  \: + \: \mathcal{O}\left(\max\left(t2^t, \, t2^{\frac{n}{m}t}\right) \right)\	
\end{eqnarray*} 
as $t \rightarrow \infty$, and so $\alpha_t \asymp  2^{(\frac{n}{m}+1)t}$. Analogously, one can show a similar result concerning the set $B_t$; that is, for any $t \in \Z_{\geq 0}$ we have $\beta_t \asymp  2^{(\frac{m}{n}+1)t}$.

To complete the proof,  we deduce that
\begin{eqnarray*}
\displaystyle\sum_{(a, b)\in \, \Nc} \, \frac{\psi(h_{a, b})}{h_{a, b}^{1-1/p}} & \asymp & \displaystyle\sum_{\substack{(a, b)\in \, \Nc \\ h_{a,b} \, = \, a^n}} \, \frac{\psi(a^n)}{a^{n(1-1/p)}} \: \: + \: \:
		\displaystyle\sum_{\substack{(a, b)\in \, \Nc \\ h_{a,b} \, = \, b^m}} \, \frac{\psi(b^m)}{b^{m(1-1/p)}} \\
		& = & \sum_{t=0}^{\infty}  \: \displaystyle\sum_{(a, b)\in \, A_t } \, \frac{\psi(a^n)}{a^{n(1-1/p)}}\: \: + \: \:\sum_{s=0}^{\infty}  \:
		\displaystyle\sum_{(a, b)\in \, B_s } \, \frac{\psi(b^m)}{b^{m(1-1/p)}} \\
		& \gg & \sum_{t=0}^{\infty}  \: \frac{\psi(2^{(t+1)n})}{2^{(t+1)n(1-1/p)}} \: \displaystyle\sum_{(a, b)\in \, A_t } \, 1 \: \: + \: \:\sum_{s=0}^{\infty}  \:
		\frac{\psi(2^{(s+1)m})}{2^{(s+1)m(1-1/p)}} \: \displaystyle\sum_{(a, b)\in \, B_s } \, 1 \\
		& = & \sum_{t=0}^{\infty}  \: \frac{\psi(2^{(t+1)n})}{2^{(t+1)n(1-1/p)}} \: \alpha_t \: \: + \: \:\sum_{s=0}^{\infty}  \:
		\frac{\psi(2^{(s+1)m})}{2^{(s+1)m(1-1/p)}} \: \beta_s \\
		& \asymp & \sum_{t=0}^{\infty}  \: \frac{2^{t+1}\psi(2^{(t+1)n})}{2^{(t+1)n(1-\frac1p-\frac1m)}}  \: \: + \: \:\sum_{s=0}^{\infty}  \: 		\frac{2^{s+1}\, \psi(2^{(s+1)m})}{2^{(s+1)m(1-\frac1p-\frac1n)}} \\
		& \asymp & \sum_{a=1}^{\infty}  \: \frac{\psi(a^{n})}{a^{n(1-\frac1p-\frac1m)}}  \: \: + \: \:\sum_{b=1}^{\infty}  \: 		\frac{\psi(b^{m})}{b^{m(1-\frac1p-\frac1n)}} \\
				& \asymp & \sum_{\substack{(a, b)\in \, \N^2 \\ h_{a, b} \, = \, a^n }}  \: \frac{\psi(h_{a, b})}{h_{a, b}^{1-1/p}}  \: \: + \: \:\sum_{\substack{(a, b)\in \, \N^2 \\ h_{a, b} \,= \, b^m }}  \: 		\frac{\psi(h_{a, b})}{h_{a, b}^{1-1/p}} \\
				& \asymp & \sum_{\substack{(a, b)\in \, \N^2 }}  \: \frac{\psi(h_{a, b})}{h_{a, b}^{1-1/p}}.  \					
\end{eqnarray*}
Thus, if the first sum converges then so does the last and the proposition is proven by a contrapositive argument.

\subsubsection{Some auxiliary lemmata and the general strategy}
\label{sec:lemmas}

For the most part our method for proving the bulk of Theorems 1a and 1b will adhere to the general strategy outlined in~\cite{BDoKL08}. Indeed, the basis of our proof will be the following consequence of Lebesgue's density theorem. 
\begin{lem}
	Let $\Omega$ be an open subset of $\R^k$ and let $E$ be a Borel subset of $\R^k$. If there exist strictly positive constant $r_0$ such that for any ball $B$ in $\Omega$ of radius $r(B)<r_0$ we have 
\begin{equation}\label{eqn:density}
\infabs{E \cap B} \: \gg \: \infabs{B},
\end{equation}
where the implied constant is independent of $B$, then $E$ has full measure in $\Omega$.	
\end{lem}

Remark: We have previously insisted that all implied constants in the Vinogradov symbols may depend only on $n,m$ and $p$, so the condition above of independence from $B$ may seem obsolete here. However, we include it in our statement as the reader may find other uses for this variant of the Lebesgue density theorem.

For technical reasons, we will take $\Omega$ to be the set $\left[\epsilon, 1\right]^2$ for some arbitrarily small $\epsilon>0$. Then upon setting $E=V_{n,m}^p(\psi)$ and $r_0=\epsilon$, where as before
$$
V_{n,m}^p(\psi)\, = \, \left\{(x, y) \in [0,1)^2: \, (x,y) \in \ell_{a,b} \text{ for infinitely many pairs } (a,b) \in \Nc \right\},
$$
it follows subject to proving (\ref{eqn:density}) that the set $V_{n,m}^p(\psi) \cap \, \left[\epsilon, 1\right]^2$ has measure $(1-\epsilon)^2$. A proof of Theorems 1a and 1b will follow upon letting $\epsilon \rightarrow 0$.

The key to establishing (\ref{eqn:density}) with $E=V_{n,m}^p(\psi)$ and $r_0=\epsilon$ will be the following lemma. 
\begin{lem} \label{lem:quasidef}
Let $E_t$ be a sequence of measurable sets which are \textbf{quasi-independent on average}; that is, the sequence $E_t$ satisfies 
\begin{equation}\label{eqn:diverge}
\sum_{t=1}^\infty \: \infabs{E_t} \, = \infty
\end{equation}
and there exists some strictly positive constant $\alpha$ for which
$$
\sum_{s,t=1}^Q \: \infabs{E_s \, \cap \, E_t} \: \leq \: 
\frac 1\alpha \, \left( \sum_{t=1}^Q \: \infabs{E_t} \right)^2
$$
for infinitely many $Q \in \N$. Then, $$\infabs{\limsup_{t \rightarrow \infty} E_t} \: \geq \: \alpha$$.
\end{lem}
\begin{proof}
	The lemma follows immediately from a generalisation of the `divergent part' of the Borel-Cantelli lemma (e.g., \cite{Spr_79} Lemma $5$), which states that for any sequence of measurable sets $E_t$ satisfying (\ref{eqn:diverge}) we have
$$	
\infabs{\limsup_{t \rightarrow \infty} E_t} \: \geq \: \limsup_{Q \rightarrow \infty} 
\frac{\left( \sum_{t=1}^Q \: \infabs{E_t} \right)^2}{\sum_{s,t=1}^Q \: \infabs{E_s \, \cap \, E_t}}.
$$
\end{proof}

The remainder of this section will be dedicated to proving the following proposition, which establishes that our strips are \textit{locally} quasi-independent
on average.

\begin{prop}\label{prop:quasi}
For any ball $B \in \left[\epsilon, 1\right]^2$ of radius $r < \epsilon$, the sequence of sets $\left\{\ell_{a,b} \, \cap \, B\right\}_{(a,b) \in \Nc}$ is quasi-independent on average. In particular,
\begin{equation}\label{eqn:condition1}
\sum_{(a,b) \in \Nc}\: \infabs{\ell_{a,b} \, \cap \, B} \, = \infty
\end{equation}
and 
\begin{equation}\label{eqn:condition2}
\sum_{\substack{(a_1, b_1) \neq (a_2, b_2)\in \, \Nc \\  h_1 \, \leq \, H, \: h_2 \, \leq \, H }}  
	\infabs{\ell_{a_1,b_1} \cap \ell_{a_2,b_2} \cap B} \: \ll \: 
\frac 1{\infabs{B}} \left( \sum_{(a, b) \in \, \Nc, \: h_{a,b} \, \leq \, H} \:
\infabs{\ell_{a,b}\cap B} \right)^2,
\end{equation}
for infinitely many $H \in \N$.
\end{prop}

By construction we have
$$
\limsup_{h_{a,b} \: \rightarrow \: \infty}\left(\ell_{a,b} \, \cap \, B\right) \: = \: V_{n,m}^p(\psi) \, \cap \, B.
$$ 
In view of Lemma \ref{lem:quasidef} and the above discussion, Proposition \ref{prop:quasi} implies that (\ref{eqn:density}) holds with $E=V_{n,m}^p(\psi)$ and $r_0=\epsilon$, and in turn that both Theorems 1a and 1b hold.

\subsection{Estimating the measure of $\boldsymbol{\ell_{a,b} \cap B}$} \label{sec:intersection}
In order to prove Proposition \ref{prop:quasi} we wish to estimate the measure of the intersections $\ell_{a,b} \cap B$ for $(a, b) \in \Nc$. We first verify condition (\ref{eqn:condition1}). To do this, for each $(a, b) \in \Nc$ we first require to estimate the number of integers $c$ for which $\ell_{a,b}(c) \cap B \neq \emptyset$.

Fix the two integers $a$ and $b$ and set $h=h_{a,b}$. If $\ell_{a,b}(c) \cap B \neq \emptyset$ then there exists $(x,y) \in B$ such that $\infabs{a^nx+b^my-c^p}< \psi(h)$. If $h_{a,b}$ is sufficiently large then we may assume $\psi(h) < \epsilon$ and so
$$
c^p \, < \,  a^nx+b^my + \psi(h) \, < a^n+b^m + \epsilon \ < \, 2h+1.
$$
Consequently, we must have that $c < 3h^{1/p}$. On the other hand,
$$
c^p \, > \,  a^nx+b^my - \psi(h) \, > \epsilon(a^n+b^m) - \epsilon \ > \, \epsilon(h-1),
$$
and so 
\begin{equation}\label{eqn:cbounds}
\frac{\epsilon^{1/p}}{2}h^{1/p} \, < \, c \, < \, 3h^{1/p}.
\end{equation}

For ease of notation, for each natural number $c$ let
$$
R_{a,b}(c) : = \left\{(x,y) \in \left[\epsilon, 1\right]^2: \, a^nx+b^my-c^p=0\right\}
$$
denote the intersection of the line $a^nx+b^my-c^p=0$ in $\R^2$ with $\left[\epsilon, 1\right]^2$. Then, $\ell_{a,b}(c) \cap B \neq \emptyset$ if and only if one of the following situations arises:
\begin{enumerate}[\quad 1.]
	\item \, $R_{a,b}(c)$ does not intersect $B$ but passes within a certain small neighbourhood of it. To be precise, the shortest distance from the line $R_{a,b}(c)$ to the centre $(x_0, y_0)$ must not exceed $2\psi(h)/\sqrt{a^{2n}+b^{2m}} + r$.
		\item \, $\,R_{a,b}(c) \cap B \neq \emptyset$. 
	\end{enumerate}
 It is clear that for a fixed pair $(a,b)$ there are at a most two distinct values of $c$ for which the former case is satisfied. To estimate how many the latter contributes we proceed as follows. Denote by $(x_0, y_0)$ the centre of the ball $B$. Then $R_{a,b}(c) \cap B \neq \emptyset$ if and only if there exists $(x,y)\in R_{a,b}(c)$ which can be written in the form
$$
x=x_0 + tr \cos\theta, \, y= y_0 + tr \sin\theta,  \quad \text{ for some } t \in [0,1) \text{ and } \theta \in [0,2\pi).
$$  
This holds if and only if 
$$
a^nx_0 + b^my_0 - r \sqrt{a^{2n}+b^{2m}} \: < \: c^p \: < \: a^nx_0 + b^my_0 + r \sqrt{a^{2n}+b^{2m}}. 
$$
Since we have assumed the radius $r < \epsilon$ and that $x_0, y_0 \geq \epsilon$, the quantity on the left hand side above is strictly positive. Therefore, all possible choices for the integer $c$ for which $R_{a,b}(c) \cap B \neq \emptyset$ lie in an interval of length $n_{a,b, B}$ satisfying
\begin{eqnarray}
n_{a,b, B} & = & r(a^{2n}+b^{2m})^{\frac{1}{2p}} \left(\frac{1}{r}\left(\frac{a^nx_0 + b^my_0}{\sqrt{a^{2n}+b^{2m}}} \: + \: r\right)^{\frac{1}{p}}  \: - \: \frac{1}{r}\left(\frac{a^nx_0 + b^my_0}{\sqrt{a^{2n}+b^{2m}}} \: - \: r\right)^{\frac{1}{p}} \right) \nonumber \\
& = & r(a^{2n}+b^{2m})^{\frac{1}{2p}} \left(\frac{ \left(\frac{a^nx_0 + b^my_0}{\sqrt{a^{2n}+b^{2m}}} \: + \: r\right)^{\frac{2}{p}}  \: - \: \left(\frac{a^nx_0 + b^my_0}{\sqrt{a^{2n}+b^{2m}}} \: - \: r\right)^{\frac{2}{p}}}{ \left(\frac{a^nx_0 + b^my_0}{\sqrt{a^{2n}+b^{2m}}} \: + \: r\right)^{\frac{1}{p}} \: + \: \left(\frac{a^nx_0 + b^my_0}{\sqrt{a^{2n}+b^{2m}}} \: - \: r\right)^{\frac{1}{p}}} \right) \frac{1}{r}\,. \label{bracket}\
\end{eqnarray}
If $p=1$ then $n_{a,b, B}= 2r(a^{2n}+b^{2m})^{\frac{1}{2}}$. Otherwise, by utilising the assumptions that $0 < r< \epsilon$ and $\epsilon \leq x_0, y_0 \leq 1$ and the trivial inequality
$$
\frac{a^n+b^m}{2} \, \leq \, \sqrt{a^{2n}+b^{2m}} \, \leq \, a^n+b^m,
$$
the denominator of the bracketed term in \eqref{bracket} is easily seen to satisfy
\begin{eqnarray*}
	\epsilon^{\frac{1}{p}} \: < \:(\epsilon+r)^{\frac{1}{p}} +(\epsilon-r)^{\frac{1}{p}} & \leq & 
	\left(\frac{a^nx_0 + b^my_0}{\sqrt{a^{2n}+b^{2m}}} \: + \:
	 r\right)^{\frac{1}{p}} \: + \: \left(\frac{a^nx_0 + 	  b^my_0}{\sqrt{a^{2n}+b^{2m}}} 
	 \: - \: r\right)^{\frac{1}{p}} \\
	 &  \leq &  (2+r)^{\frac{1}{p}} +(2-r)^{\frac{1}{p}} \quad  \leq  \: 4.\
\end{eqnarray*}
When $p=2$ the numerator of the bracketed term in \eqref{bracket} is simply equal to $2r$. 
For $p \geq 3$ one can utilise the Mean Value Theorem (or the Generalised Binomial Theorem)  to quickly verify the numerator is $\asymp r$, where the implied constant depends only upon $\epsilon$. 
Combining the two cases ($1.$ and $2.$) outlined earlier in this subsection we conclude that the number of possible choices for $c$ for which $\ell_{a,b}(c) \cap B \neq \emptyset$ must be $\asymp rh^{1/p}$.

We may now estimate the measure of the intersection $\ell_{a,b} \cap B$. For each integer $c$ we have the trivial upper bound 
$$
\infabs{\ell_{a,b}(c) \cap B} \, \leq \, \frac{4r\psi(h)}{\sqrt{a^{2n}+b^{2m}}}\, \ll \frac{r\psi(h)}{h}.
$$
However, we have no general lower bound on the $\ell_{a,b}(c) \cap B$ as the intersection may be even as small as a single point. To counter this problem we consider a subset of those integers $c$ satisfying $\ell_{a,b}(c) \cap B \neq \emptyset$ for which this intersection is sufficiently large. 
Let $\frac{1}{2}B$ be the ball $B$ scaled by one half; that is, $\frac{1}{2}B$ is the open ball in $\Omega$ with centre $(x_0, y_0)$ and radius $r/2$. It is easy to see that for $h$ sufficiently large that if $\ell_{a,b}(c)$ intersects $\frac{1}{2}B$ then $\infabs{\ell_{a,b}(c) \cap B} \, \geq \, 2r\psi(h)/\sqrt{a^{2n}+b^{2m}} \asymp r\psi(h)/h$. As before, the number of possible choices of $c$ for which $\ell_{a,b}(c) \cap \frac{1}{2}B \neq \emptyset$ is $\asymp rh^{1/p}$. 

One should note that the upper bound obtained for the number of choices of $c$ in the calculations above coincides with the trivial one (that is, the diameter of the ball divided by the maximum distance between two adjacent lines as will be calculated in \S\ref{sec:intersections}). However, it is the lower bound which is of importance to us in proving the main theorem.  

Combining the upper and lower bounds for $\infabs{\ell_{a,b}(c) \cap B}$ and the estimates for the number of $c$ for which these intersections are non-empty yield that
$$
r^2\frac{\psi(h)}{h^{1-1/p}}   \, = \, rh^{1/p} \cdot r\frac{\psi(h)}{h} \, \ll \, \infabs{\ell_{a,b} \cap B} \, \ll \, rh^{1/p} \cdot r\frac{\psi(h)}{h} \, = \, r^2\frac{\psi(h)}{h^{1-1/p}}.
$$
In other words, for any $(a, b)\in\N^2$ and any open ball $B \subset \Omega$ we have that
\begin{equation}\label{eqn:intersectbound}
 \infabs{\ell_{a,b} \cap B} \quad \asymp \quad \infabs{B}\frac{\psi(h_{a,b})}{h_{a,b}^{1-1/p}}.
\end{equation}
Thus, condition (\ref{eqn:condition1}) holds for the sequence of sets $\ell_{a,b} \cap B$.

\subsection{Estimating the measure of $\boldsymbol{\ell_{a_1,b_1} \cap \ell_{a_2,b_2} \cap B}$} \label{sec:intersections}
All that remains is to establish condition (\ref{eqn:condition2}). 
Fix two distinct pairs of integers $(a_1,b_1)$ and $(a_2,b_2)$ in $\Nc$ and for ease of notation set $h_1:=h_{a_1, b_1}$ and $h_2:=h_{a_2, b_2}$. Recall that for $(a, b) \in \Nc$ and $c \in \Z_{\geq 0}$ our notation
$$
R_{a,b}(c) : = \left\{(x,y) \in \left[\epsilon, 1\right]^2: \, a^nx+b^my-c^p=0\right\}.
$$
A consequence of the assumption that our natural numbers $a$ and $b$ are coprime is that for any $c_1$ and $c_2$ the line segments $R_{a_1, b_1}(c_1)$ and  $R_{a_2, b_2}(c_2)$ cannot be parallel. As we will see, this ensures that the intersection of the strips $\ell_{a_1,b_1}$ and $\ell_{a_2,b_2}$ is not too large. For the remainder of the section we denote by $\alpha:=\alpha(a_1, b_1, a_2, b_2)$ the strictly positive acute angle between the lines defining $R_{a_1, b_1}(c_1)$ and  $R_{a_2, b_2}(c_2)$; that is, the angle between the vectors $(a_1^n, b_1^m)$ and $(a_2^n, b_2^m)$ in $\R^2$.

Given a ball $B$ in $\left[\epsilon, 1\right]^2$ we wish to deduce bounds on the size of the intersection $\ell_{a_1,b_1} \cap \ell_{a_2,b_2} \cap B$. To do this, we first fix an integer $c_1$ and then  estimate the size of each intersection $\ell_{a_1,b_1}(c_1) \cap \ell_{a_2,b_2} \cap B$. Firstly, observe that the set $\ell_{a_1,b_1}(c_1) \cap B$ can be covered by a strip of length $2r$ and of width $\frac{2\psi(h)}{\sqrt{a_1^{2n}+b_1^{2m}}}$. This strip is a section of the $\frac{\psi(h)}{\sqrt{a_1^{2n}+b_1^{2m}}} \, $-neighbourhood of the line $a_1^nx+b_1^my-c_1^p=0$.
We now consider the size of the intersection of $\ell_{a_2,b_2}$ with such a strip. 

The set $\ell_{a_2,b_2}$ consists of a collection of neighbourhoods of parallel lines of the form
$$
a_2^nx+b_2^my-c_2^p=0, \quad c_2 \geq 0,
$$
in $\left[\epsilon, 1\right]^2$; that is, neighbourhoods of line segments $R_{a_2, b_2}(c_2)$ for $c_2 \geq 0$. A simple geometric argument combined with the binomial theorem shows that the distance between any two such adjacent line segments, say $R_{a_2, b_2}(c_2)$ and $R_{a_2, b_2}(c_2+1)$, is given by
\begin{eqnarray*}
\frac{\left(\frac{(c_2+1)^p-c_2^p}{b_2^m}\right)\left(\frac{(c_2+1)^p-c_2^p}{b_2^m}\right)}{\sqrt{\left(\frac{(c_2+1)^p-c_2^p}{b_2^m}\right)^2+\left(\frac{(c_2+1)^p-c_2^p}{b_2^m}\right)^2}} & = & 
\frac{\left((c_2+1)^p-c_2^p\right)\left((c_2+1)^p-c_2^p\right)}{\sqrt{a_2^{2n}\left((c_2+1)^p-c_2^p\right)^2+b_2^{2m}\left((c_2+1)^p-c_2^p\right)^2}}  \\ 
& \asymp & \frac{(c_2+1)^p-c_2^p}{h_2} \quad \asymp \quad \frac{c_2^{p-1}}{h_2}. \ 
\end{eqnarray*}
By the set of inequalities (\ref{eqn:cbounds}) we know for the intersection $\ell_{a_2,b_2}\cap B$ to be non-empty that $c_2 \asymp h_2^{1/p}$, and so this distance is $\asymp h_2^{\frac{p-1}{p} -1} = h_2^{-1/p}$. Therefore, if two adjacent line segments of the form $R_{a_2, b_2}(c_2)$ and $R_{a_2, b_2}(c_2 \pm 1)$  intersect the line segment $R_{a_1, b_1}(c_1)$ then the distance between the two intersection points is $\asymp (h_2^{1/p}\sin\alpha)^{-1}$. In turn, this implies that there are $\ll r h_2^{1/p}\sin\alpha +1$ non-empty intersections of this type affecting the ball $B$ for each fixed natural number~$c_1$. Furthermore, for each $c_2$ the set $\ell_{a_1,b_1}(c_1) \cap \ell_{a_2,b_2}(c_2)$ takes the form of a parallelepiped with volume $\asymp \frac{\psi(h_1)\psi(h_2)}{h_1h_2\sin\alpha}$. It follows that 
$$
\infabs{\ell_{a_1,b_1}(c_1) \cap \ell_{a_2,b_2} \cap B} \: \ll \:\frac{ (r h_2^{1/p}\sin\alpha +1)\psi(h_1)\psi(h_2)}{h_1 \, h_2\sin\alpha}.
$$

Finally, we recall from \S\ref{sec:intersection} that there are $\asymp r h_1^{1/p}$ possible choices of $c_1$ for which $\ell_{a_1,b_1}(c_1) \cap B \neq \emptyset$, and so
\begin{eqnarray}
	\infabs{\ell_{a_1,b_1} \cap \ell_{a_2,b_2} \cap B} & \ll & \frac{ (r h_2^{1/p}\sin\alpha +1)\psi(h_1)\psi(h_2) \cdot r h_1^{1/p}}{h_1h_2\sin\alpha} \notag  \\
	& \asymp &  \infabs{B}	\frac{\psi(h_1)\psi(h_2)}{h_1^{1-1/p} \,h_2^{1-1/p}} \: \left( 1 + \frac{1}{rh_2^{1/p}\sin\alpha} \right). \label{eqn:sumsize} \
\end{eqnarray}
We now split our calculation into three exhaustive subcases depending on the size of the angle $\alpha$.

\subsubsection{Large angle}

First, we assume that $\alpha$ is large enough to satisfy 
\begin{equation}\label{eqn:large}
\sin\alpha \geq \frac{1}{rh_2^{1/p}}. 
\end{equation}
It immediately follows from (\ref{eqn:intersectbound}) and (\ref{eqn:sumsize}) that
$$
	\infabs{\ell_{a_1,b_1} \cap \ell_{a_2,b_2} \cap B} \: \ll \: \infabs{B}	\frac{\psi(h_1)\psi(h_2)}{h_1^{1-1/p} \,h_2^{1-1/p}},
$$
and so
\begin{eqnarray*}
	\sum_{\substack{(a_1, b_1) \neq (a_2, b_2)\in \, \Nc \\  \text{satisfying (\ref{eqn:large})} \\ h_1 \, \leq \, H, \: h_2 \, \leq \, H}}  
	\infabs{\ell_{a_1,b_1} \cap \ell_{a_2,b_2} \cap B} & \ll & 
\infabs{B} \sum_{\substack{(a_1, b_1) \in \, \Nc  \\   h_1 \, \leq \, H}} \frac{\psi(h_1)}{h_1^{1-1/p} }  
\sum_{\substack{(a_2, b_2) \in \, \Nc \\   h_2 \, \leq \, H}} \frac{\psi(h_2)}{h_2^{1-1/p} } \\
& \asymp & \frac{1}{\infabs{B}} \left( \sum_{\substack{(a, b) \in \, \Nc, \: h_{a,b} \, \leq \, H}}
\infabs{\ell_{a,b}\cap B} \right)^2.\
\end{eqnarray*}
Thus, the set of pairs $(a_1, b_1),(a_2, b_2) \in \Nc$ with property (\ref{eqn:large}) satisfy condition~(\ref{eqn:condition2}).

\subsubsection{Medium angle cases}

For ease of notation, in the remainder of the proof we write $N:=\max(n, m)$, $M:=\min(n, m)$ and $K:=\gcd(n,m)$. 

We next consider the case when the angle between the line segments $R_{a_1, b_1}(c_1)$ and  $R_{a_2, b_2}(c_2)$ is of intermediate size. In this case the intersection $\ell_{a_1,b_1}(c_1) \cap \ell_{a_2,b_2}(c_2) \cap B$ may be quite large and contribute more than its fair share to the volume sum we are calculating. However, we show that the number of pairs $(a_1, b_1)$ and $(a_2, b_2)$ satisfying both this property is sufficiently small. We will require to partition the medium angle cases. To be precise, for any $Q\geq 0$ define the partial sums
$$
	R(Q) = \sum_{q=0}^{Q} \left(\frac{K}{M}\right)^q \quad\quad \text{and} \quad\quad S(Q) = \sum_{q=0}^{Q} \left(\frac{K}{N}\right)^q,
$$
and assume that
\begin{equation}\label{eqn:midangle}
	\frac{1}{r^{R(Q)} \, h_2^{S(Q)/p}} \: \geq \: \sin\alpha \: \geq \: \frac{1}{r^{R(Q+1)} \, h_2^{S(Q+1)/p}}.
\end{equation}
Note that $R(0)=S(0)=1$ and so the upper bound when $Q=0$ coincides with the lower bound of the large angle case. We now prove condition~(\ref{eqn:condition2}) holds for any given $Q\geq 0$.

Fix $Q$. It immediately follows from equations (\ref{eqn:sumsize}) and (\ref{eqn:midangle}) that
\begin{equation}\label{eqn:intersection}
	\infabs{\ell_{a_1,b_1} \cap \ell_{a_2,b_2} \cap B} \: \ll \:  r	\frac{\psi(h_1)\psi(h_2)}{h_1^{1-\frac1p} \,h_2 \, \sin\alpha} \: \leq \: 
	 r^{1+R(Q+1)}	\frac{\psi(h_1)\psi(h_2)}{h_1^{1-\frac1p} \,h_2^{1-S(Q+1)/p}}.
\end{equation}
We wish to find an upper bound for the number of quadruples $(a_1, b_1, a_2, b_2)$ satisfying~(\ref{eqn:midangle}). To begin with, observe that the natural numbers $a_1, b_1, a_2$ and $b_2$ satisfy
$$
	a_1^nb_2^m-a_2^nb_1^m \: = \: 
	\left( a_1^{\frac{n}{K}}b_2^{\frac{m}{K}}-a_2^{\frac{n}{K}}b_1^{\frac{m}{K}} \right) 
	\sum_{t=1}^K a_1^{n(1-\frac tK)}b_1^{\frac{m(t-1)}{K}} a_2^{\frac{n(t-1)}{K}} b_2^{m(1-\frac tK)}.
$$
For $K=1$ this statement is trivial and the sum on the right hand side is trivial, so for now assume that $K\geq2$. In view of the defining properties of the set $\Nc$, for  $i=1,2$ and for any $t=2, \ldots, K$ we have 
$$
\left(\frac{1}{2} \right)^{1-1/K} \: \leq \: \left(\frac{1}{2} \right)^{(t-1)/K} \: \leq \: \frac{a_i^{n(t-1)/K}}{b_i^{m(t-1)/K}} \: \leq \: 2^{1-1/K} \: \leq \: 2^{1-1/K}
$$
and for $t=1, \ldots, K-1$ we have
$$
\left(\frac{1}{2} \right)^{2-1/K} \: \leq \: \left(\frac{1}{2} \right)^{1+t/K} \: \leq \: \frac{a_i^{n(1-t/K)}}{b_i^{m(1-t/K)}} \: \leq \: 2^{1+t/K}\: \leq \: 2^{2-1/K}.
$$
It follows that 
\begin{eqnarray*}
	\sum_{t=1}^K a_1^{n(1-t/K)}b_1^{m(t-1)/K} a_2^{n(t-1)/K} b_2^{m(1-t/K)} & \asymp & 
	\sum_{t=1}^K{h_1^{(1-t/K)+(t-1)/K}}h_2^{(1-t/K)+(t-1)/K} \\ & = & \sum_{t=1}^K{h_1^{1-1/K}}h_2^{1-1/K} \quad \asymp \quad h_1^{1-1/K}h_2^{1-1/K}. \	
\end{eqnarray*} 
Furthermore, since $\gcd(a_1, b_1)=\gcd(a_2, b_2)=1$ it is certain that $a_1^{\frac{n}{K}}b_2^{\frac{m}{K}}-a_2^{\frac{n}{K}}b_1^{\frac{m}{K}}$ is non-zero. In what follows, we denote by $\norm{M}$ the absolute value of the determinant of a matrix $M$.  As $\alpha$ is precisely the positive (acute) angle between the two vectors $(a_1^n, b_1^m)$ and $(a_2^n, b_2^m)$, the cross product formula then readily implies that 
\begin{eqnarray}
1 \quad \leq \quad \begin{Vmatrix}
  a_1^{n/K} & b_1^{m/K} \\
  a_2^{n/K} & b_2^{m/K} 
 \end{Vmatrix}
 & \asymp  & h_1^{1/K-1}h_2^{1/K-1} \begin{Vmatrix}
  a_1^{n} & b_1^{m} \\
  a_2^{n} & b_2^{m} 
 \end{Vmatrix} \notag \\
 & = & h_1^{1/K-1}h_2^{1/K-1} \sin\alpha \sqrt{a_1^{2n}+b_1^{2m}}\sqrt{a_2^{2n}+b_2^{2m}} \notag \\ 
 & \asymp & h_1^{1/K}h_2^{1/K}\sin\alpha. \label{eqn:determinant} \
\end{eqnarray} 
\begin{rem}
	One can easily verify that this system of inequalities also holds in the case $K=1$ and is sufficient for the proof in the classical $n=m=1$ case. However, for a generic coprime $m$ and $n$ we do not gain any new information from (\ref{eqn:determinant}), and this provides an obstacle in extending our results to the $K=1$ case in general.
\end{rem}

In particular, for any $K$, assume that we have two pairs $(a_1, b_1)$ and $(a_2, b_2)$ both in $\Nc$ and both satisfying~(\ref{eqn:midangle}). It follows that 
\begin{equation} \label{eqn:count}
	\infabs{b_2^{\frac{m}{K}}-a_1^{-\frac{n}{K}}a_2^{\frac{n}{K}}b_1^{\frac{m}{K}}}  \: 
	\ll \: r^{-R(Q)}a_1^{-\frac{n}{K}}h_1^{\frac1K}\: h_2^{\frac1K-S(Q)/p} \: \ll \: r^{-R(Q)}h_2^{\frac1K-S(Q)/p}.
\end{equation}

Now, assume that the pair $(a_1, b_1)$, and therefore $h_1$, is fixed. Then the above calculation yields that for each fixed $a_2$ there are at most a constant times $r^{-KR(Q)/m}h_2^{1/m-KS(Q)/(mp)}$ possible choices for $b_2$. Similarly, if one were to fix $b_2$ then we would have at most a constant times $r^{-K R(Q)/n}h_2^{1/n-K S(Q)/(np)}$ ways of choosing $a_2$. Note that unless $S(Q) \leq \frac{p }{K}$ there are no such choices for sufficiently large $h_2$ and so without loss of generality we may assume this inequality holds for each $Q$.

We now calculate the total volume that the intersections $\ell_{a_1,b_1} \, \cap \, \ell_{a_2,b_2} \, \cap \, B$ contribute to our measure sum in the case that (\ref{eqn:midangle}) holds. For conciseness of notation define the function $f:  \N^2\times\left\{n,m\right\} \rightarrow \N$ in the following way:
$$
f(a,b,i)= \begin{cases}
			a^{n}, & i = n. \\
 		& \\
  		b^{m}, & i = m.
  	\end{cases}
$$
Without loss of generality we assume that $h_2 \geq h_1$. 
Following on from (\ref{eqn:intersection}) we have that
\begin{eqnarray*}
	\sum_{\substack{(a_1, b_1) \neq (a_2, b_2)\in \, \Nc \\  \text{satisfying (\ref{eqn:midangle})} \\ h_1 \, \leq \,h_2 \,\leq  \, H}} 
	\infabs{\ell_{a_1,b_1} \cap \ell_{a_2,b_2} \cap B} & \ll & 
	r^{1+R(Q+1)} \sum_{\substack{(a_1, b_1) \neq (a_2, b_2)\in \, \Nc \\  \text{satisfying (\ref{eqn:midangle})} \\ h_1 \, \leq \,h_2 \,\leq  \, H}}  
	\: 		\frac{\psi(h_1)\psi(h_2)}{h_1^{1-\frac1p} \,h_2^{1-\frac {S(Q+1)}{p}}} \\
	& \leq &  r^{1+R(Q+1)} \left( S_{n,n} \: + \: S_{n,m} \: + \: S_{m,n} \: + \: S_{m,m} \right), \
\end{eqnarray*}
where 
$$
	S_{i,j}: \: = \: \sum_{g_2=1}^{\left\lfloor H^{1/j} \right\rfloor} 
	\sum_{g_1=1}^{\left\lfloor g_2^{j/i}\right\rfloor} \: 
	\sum_{\substack{(a_1, b_1) \neq (a_2, b_2)\in \, \Nc \\  \text{satisfying (\ref{eqn:midangle})} \\ h_1=f(a_1, b_1, i)=g_1^i, \\ h_2=f(a_2, b_2, j)=g_2^j 
}}  
	\: 	\frac{\psi(h_1)\psi(h_2)}{h_1^{1-\frac1p} \,h_2^{1-\frac {S(Q+1)}{p}}}.
$$
 The somewhat cumbersome collection of sum conditions in the final expression serve a very important purpose. We have split our volume calculation into four smaller sums $S_{i,j}$ (for $i,j \in \left\{n,m\right\}$). Each sum corresponds to the pairs of natural numbers $(a_1, b_1)$ and $(a_2, b_2)$ for which $h_1$ takes the form of an $i$'th power and $h_2$ takes the form of a $j$'th power. However, it is for example perfectly possible for the height $h_1$ to take the form of an $n$-th power yet to be attained on the second component; i.e., we could have for some natural number $b$ that $a_1^{n/m} < b_1=b^n$ and so $h_1=b_1^m=(b^n)^m=(b^m)^n$. If $(a_1, b_1)$ does take this form then we do not wish to count in the sums $S_{n,n}$ or $S_{n,m}$ the quadruple $(a_1, b_1, a_2, b_2)$ as they will have already appeared in the sum $S_{m,n}$ (if $h_2=a_2^n$) or $S_{m,m}$ (if $h_2=b_2^m$). The function $f$ guarantees that the values taken by $h_1$ and $h_2$ genuinely appear in the first component if they appear as $g_1^n$ or $g_2^n$ and the second component if they appear as $g_1^m$ or~$g_2^m$. This painstaking stipulation will prove important in our calculation.

 For the purpose of clarity we will exhibit the calculations relating to the two sums $S_{n,n}$ and $S_{m,n}$ separately; upper bounds for the remaining sums $S_{m,m}$ and $S_{n,m}$ follow analogously.

Firstly, the case $i=j=n$ corresponds to those pairs $(a_1, b_1)$ and $(a_2, b_2)$  for which $h_1=a_1^n$ and $h_2=a_2^n$. We use the crude upper bound of $h_1^{1/m}$ to estimate the number of possible pairs $(a_1, b_1)$ satisfying $h_1=a_1^n$; it turns out that taking into account conditions (\ref{eqn:restrictions}) to reduce this trivial bound serves little purpose here. However, once $a_2$ is also chosen we may use our estimate for the number of ways of choosing the integer $b_2$ so that (\ref{eqn:midangle}) holds once the other three natural numbers $a_1, a_2, b_1$ are already prescribed.  For fixed natural numbers $g_1$ and $g_2$ with $g_1 \leq g_2$ we have
$$
	\sum_{\substack{(a_1, b_1) \neq (a_2, b_2)\in \, \Nc \\  \text{satisfying (\ref{eqn:midangle})} \\  h_1=a_1^n=g_1^n, \\ h_2=a_2^n=g_2^n }} \: 1 \quad \ll \quad g_1^{\frac nm} \cdot r^{-\frac {KR(Q)}{m}}g_2^{\frac nm (1-KS(Q)/p)}.
$$
Also, observe that $S(Q) =  \frac NK (S(Q+1)-1)$ and so
$$
S(Q+1) - \frac{KS(Q)}{m} \quad = \quad S(Q+1) - \frac{N}{m}(S(Q+1) - 1) \quad \leq \quad 1.
$$
Therefore,
\begin{eqnarray*}
	\sum_{\substack{(a_1, b_1) \neq (a_2, b_2)\in \, \Nc \\  \text{satisfying (\ref{eqn:midangle})} \\\\ h_1=a_1^n=g_1^n, \\ h_2=a_2^n=g_2^n}}  
	\: 	\frac{\psi(h_1)\psi(h_2)}{h_1^{1-\frac1p} \,h_2^{1-S(Q+1)/p}}
	& = & \frac{\psi(g_1^n)\psi(g_2^n)}{g_1^{n(1-\frac1p)} \,g_2^{n(1-\frac {S(Q+1)}{p})}} \:
		\sum_{\substack{(a_1, b_1) \neq (a_2, b_2)\in \, \Nc \\  \text{satisfying (\ref{eqn:midangle})} \\  h_1=a_1^n=g_1^n, \\ h_2=a_2^n=g_2^n }} \: 1 \\
		& \ll & r^{-\frac {KR(Q)}{m}} \frac{\psi(g_1^n)\psi(g_2^n)}{g_1^{n(1-\frac 1p-\frac1m)}
		g_2^{n(1-\frac{S(Q+1)}{p}-\frac1m + \frac{KS(Q)}{mp})}} \\
		& \leq & 		r^{-\frac {KR(Q)}{m}} \frac{\psi(g_1^n)}{g_1^{n(1-\frac 1p-\frac1m)}}\frac{\psi(g_2^n)}{g_2^{n(1-\frac 1p-\frac1m)}}, \
\end{eqnarray*}

and so $S_{n,n}$ is bounded above by
$$
 r^{-\frac {KR(Q)}{m}}\sum_{g_2=1}^{\left\lfloor H^{1/n} \right\rfloor} 
	\sum_{g_1=1}^{g_2} \frac{\psi(g_1^n)}{g_1^{n(1-\frac 1p-\frac1m)}}\frac{\psi(g_2^n)}{g_2^{n(1-\frac 1p-\frac1m)}} \quad \leq \quad r^{-\frac {KR(Q)}{M}}\left( \sum_{a=1}^{\left\lfloor H^{1/n} \right\rfloor} 
	\frac{\psi(a^n)}{a^{n(1-\frac 1p-\frac1m)}} \right)^2.
$$

Next, in the case $i=n, \, j=m$ the sum $S_{n,m}$ corresponds to those pairs $(a_1, b_1)$ and $(a_2, b_2)$ for which $h_1=b_1^n$ and $h_2=a_2^m$. Here we use our estimate for the number of ways of choosing the integer $a_2$ so that (\ref{eqn:midangle}) is satisfied once the other three natural numbers $a_1, b_1, b_2$ are already chosen.  For any given natural numbers $g_1$ and $g_2$ with $g_1^n \leq g_2^m$ we have
$$
	\sum_{\substack{(a_1, b_1) \neq (a_2, b_2)\in \, \Nc \\  \text{satisfying (\ref{eqn:midangle})} \\  h_1=b_1^n=g_1^n, \\ h_2=a_2^m=g_2^m }} \: 1 \quad \leq \quad g_1^{\frac nm} \cdot r^{-\frac {KR(Q)}{n}}g_2^{\frac mn (1-KS(Q)/p)}.
$$
As before it follows that
\begin{eqnarray*}
	\sum_{\substack{(a_1, b_1) \neq (a_2, b_2)\in \, \Nc \\  \text{satisfying (\ref{eqn:midangle})} \\\\ h_1=b_1^n=g_1^n, \\ h_2=a_2^m=g_2^m}}  
	\: 	\frac{\psi(h_1)\psi(h_2)}{h_1^{1-\frac1p} \,h_2^{1-S(Q+1)/p}}
	& = & \frac{\psi(g_1^n)\psi(g_2^m)}{g_1^{n(1-\frac1p)} \,h_2^{m(1-\frac {S(Q+1)}{p})}} \:
		\sum_{\substack{(a_1, b_1) \neq (a_2, b_2)\in \, \Nc \\  \text{satisfying (\ref{eqn:midangle})} \\  h_1=b_1^n=g_1^n, \\ h_2=a_2^m=g_2^m }} \: 1 \\
		& \ll & r^{-\frac {KR(Q)}{n}} \frac{\psi(g_1^n)\psi(g_2^m)}{g_1^{n(1-\frac1p-\frac1m)}
		g_2^{m(1-\frac{S(Q+1)}{p}-\frac1n + \frac{KS(Q)}{np})}} \\
		& \leq & 		r^{-\frac {KR(Q)}{n}} \frac{\psi(g_1^n)}{g_1^{n(1-\frac1p-\frac1m)}}\frac{\psi(g_2^m)}{g_2^{m(1-\frac1p-\frac1n)}}, \
\end{eqnarray*}
and so
$$
S_{n,m} \quad \leq \quad r^{-\frac {KR(Q)}{M}}\left( \sum_{a=1}^{\left\lfloor H^{1/n} \right\rfloor} 
	\frac{\psi(a^n)}{a^{n(1-\frac1p-\frac1m)}} \right)\left( \sum_{b=1}^{\left\lfloor H^{1/m} \right\rfloor} 
	\frac{\psi(b^m)}{b^{m(1-\frac1p-\frac1n)}} \right).
$$

One can carry out analogous calculations corresponding to the remaining two cases, leading to the estimates 
$$
S_{m,n} \quad \leq \quad r^{-\frac {KR(Q)}{M}}\left( \sum_{b=1}^{\left\lfloor H^{1/m} \right\rfloor} 
	\frac{\psi(b^m)}{b^{m(1-\frac1p-\frac1n)}} \right)\left( \sum_{a=1}^{\left\lfloor H^{1/n} \right\rfloor} 
	\frac{\psi(a^n)}{a^{n(1-\frac1p-\frac1m)}} \right)
$$
and
$$
S_{m,m} \quad \leq \quad r^{-\frac {KR(Q)}{M}}\left( \sum_{b=1}^{\left\lfloor H^{1/m} \right\rfloor} 
	\frac{\psi(b^m)}{b^{m(1-\frac1p-\frac1n)}} \right)^2.
$$
Thus, the sum of sums $S_{n,n} + S_{n,m} + S_{m,n} + S_{m,m}$ is bounded above by 
$$
r^{-\frac {KR(Q)}{M}}\left( \sum_{a=1}^{\left\lfloor H^{1/n} \right\rfloor} 
	\frac{\psi(a^n)}{a^{n(1-\frac1p-\frac1m)}} \: + \: \sum_{b=1}^{\left\lfloor H^{1/m} \right\rfloor} 
	\frac{\psi(b^m)}{b^{m(1-\frac1p-\frac1n)}} \right)^2
	\: \asymp \: r^{-\frac {KR(Q)}{M}} \left( \sum_{\substack{(a, b) \in \, \N^2 \\ h_{a,b} \, \leq \, H}} \frac{\psi(h_{a,b})}{h_{a,b}^{1-\frac 1p} } \right)^2.
$$
It follows from \S\ref{sec:restrictions} that this quantity is in turn bounded above by a constant times 
$$
\frac{1}{r^{4+\frac {KR(Q)}{M}}} \left( \sum_{\substack{(a, b) \in \, \Nc \\ h_{a,b} \, \leq \, H}} \infabs{B} \: \frac{\psi(h_{a,b})}{h_{a,b}^{1-1/p} } \right)^2 \quad \asymp \: \frac{1}{r^{4+\frac {KR(Q)}{M}}} \left( \sum_{\substack{(a, b) \in \, \Nc \\ h_{a,b} \, \leq  \, H }}
	\infabs{\ell_{a,b} \: \cap \: B} \right)^2.
$$
It is readily verified that $ 1+R(Q+1)- (4 + \frac {KR(Q)}{M}) = -2$
and so the total contribution to our volume sum made by pairs $(a_1, b_1)$ and $(a_2, b_2)$ for which (\ref{eqn:midangle}) holds must satisfy
$$
	\sum_{\substack{(a_1, b_1) \neq (a_2, b_2)\in \, \Nc \\  \text{satisfying (\ref{eqn:midangle})} \\ h_1 \, \leq \,h_2 \,\leq  \, H}} 
	\infabs{\ell_{a_1,b_1} \cap \ell_{a_2,b_2} \cap B} \quad \ll \quad
 \frac{1}{\infabs{B}}  \left( \sum_{\substack{(a, b) \in \, \Nc \\  h_{a,b} \, \leq  \, H}}
	\infabs{\ell_{a,b} \cap B} \right)^2.  
$$
Condition~(\ref{eqn:condition2}) therefore holds in the medium angle cases for any given $Q \geq 0$.

Notice that if $K = N$ (i.e., if $n=m$) then $S(Q)$ diverges as $Q \rightarrow \infty$ and so we may consider as small an angle as we wish via the medium angle method. Moreover, as soon as $S(Q) > \frac pK$ then there are no quadruples of integers satisfying (\ref{eqn:count}) for sufficiently large $h_2$ and so the volume contribution is zero. Thus, the proof of Theorem 1a is complete and we proceed upon the assumptions of Theorem 1b; in particular that $n\neq m$ (and so $K<N$).

\subsubsection{Small angle} Finally, we consider the scenario in which the angle $\alpha$ between the vectors $(a_1^n, b_1^m)$ and $(a_2^n, b_2^m)$ is very small indeed. Fix some sufficiently small $\epsilon>0$, which may depend upon only $n$, $m$ and $p$. By taking $Q_0$ sufficiently large such that $S(Q_0+1) \geq N (N-K)^{-1} - \epsilon p$, the cases of angles satisfying  (\ref{eqn:midangle}) for $Q \leq Q_0$ follow from the medium angle method; we may choose such an $\epsilon$ since $S(Q)$ converges to~$N (N-K)^{-1}$ when $K < N$.  We are left only with those cases in which the angle $\alpha$ satisfies

\begin{equation}\label{eqn:vsa1}
\sin\alpha \: < \: \frac{1}{r^{R(Q_0+1)} \, h_2^{\frac N {p(N-K)} - \epsilon}}.
\end{equation}
Whilst the intersection inside the ball $B$ of the corresponding strips $\ell_{a_1, b_1}$ and $\ell_{a_2, b_2}$ may now be very large due to the strips being close to parallel, we show that since $(a,b)$ are chosen only from $\Nc$ then there are very few pairs indeed satisfying this condition. 

Recall that inequality (\ref{eqn:determinant}) tells us that any two distinct pairs of natural numbers $(a_1, b_1)$ and $(a_2, b_2)$ in $\Nc$ must also satisfy $\sin\alpha \gg h_1^{-1/K}h_2^{-1/K}$. It follows from~(\ref{eqn:sumsize}) that
\begin{equation}\label{eqn:intersection1}
	\infabs{\ell_{a_1,b_1} \cap \ell_{a_2,b_2} \cap B} \: \ll \:  r	\frac{\psi(h_1)\psi(h_2)}{h_1^{1-1/p} \,h_2 \, \sin\alpha} \: \leq \: 
	r\frac{\psi(h_1)\psi(h_2)}{h_1^{1-1/p-1/K} \,h_2^{1-1/K}}.
\end{equation}

Proceeding as before, and again assuming without loss of generality that $h_2 \geq h_1$, it is quickly verified that an upper bound for the volume sum required is given by sums (for $i,j \in \left\{n,m\right\}$) of the form
$$
T_{i,j} \: = \: \sum_{t_2=1}^{\left\lfloor \log(H^{1/j})\right\rfloor} \sum_{t_1=1}^{\left\lfloor \frac{j}{i}t_2 \right\rfloor } \: \frac{\psi(2^{it_1})\psi(2^{jt_2})}{2^{it_1(1-\frac1p-\frac1K)} \,2^{jt_2(1-\frac1K)}} \:\sum_{\substack{2^{t_1} \, < \, g_1 \, \leq  \, 2^{t_1+1} \\2^{t_2} \, < \, g_2 \, \leq  \, 2^{t_2+1}
}}   \: \:  \sum_{\substack{(a_1, b_1) \neq (a_2, b_2)\in \, \N^2 \\  \text{satisfying (\ref{eqn:vsa1})}  \\ h_1=g_1^i, \, h_2=g_2^j  }}
\: 1.
$$
However, by combining inequality (\ref{eqn:determinant}) with (\ref{eqn:vsa1}) we deduce that  
 \begin{equation}\label{eqn:vsa2}
 1 \: \leq \: \begin{Vmatrix}
a_1^{\frac{n}{K}}&b_1^{\frac{m}{K}}\\a_2^{\frac{n}{K}}&b_2^{\frac{m}{K}}
\end{Vmatrix} \: \ll \:  h_1^{\frac1K}h_2^{\frac1K}\sin\alpha \: <\: \frac{h_1^{\frac{1}{K}}h_2^{\frac{1}{K}-\frac N {p(N-K)}+\epsilon}}{r^{R(Q_0+1)}}.
 \end{equation}
 We may therefore deduce that $p\geq\frac {KN} {N-K}-\epsilon'$ (where $ \epsilon' = \epsilon Kp$), else no such quadruples exist for sufficiently large $h_2$. Hence, by taking $\epsilon$ (ergo $\epsilon'$) sufficiently small and noting that $KN>(N-K)$ when $N>K$, the proof is complete for the `$p=1$' cases. 
Additionally, when $K \geq 3$ we must have $P < K < \frac {KN} {N-K}$ by  (\ref{eqn:pbound}), and so the main result holds in that case. Finally, when  $\gcd(n,m)=2$ one can check by hand that all remaining cases are either covered by the above condition or violate (\ref{eqn:pbound}). This completes the proof of Theorem~1b.

\section{Proof of Theorem 2} \label{HM}

For completeness we give a very brief introduction to Hausdorff measures and dimension. For further details see \cite{BeDo_99}.
Let
$F\subset \R^n$.
 Then, for any $\rho>0$ a countable collection $\{B_i\}$ of balls in
$\R^n$ with diameters $\mathrm{diam} (B_i)\le \rho$ such that
$F\subset \bigcup_i B_i$ is called a $\rho$-cover of $F$. For  a dimension function $f$  define
\[
\Ha_\rho^f(F)=\inf \sum_i f(\mathrm{diam}B_i),
\]
where the infimum is taken over all possible $\rho$-covers $\{B_i\}$ of $F$. It is easy to see that $\Ha_\rho^f(F)$ increases as $\rho$ decreases and so approaches a limit as $\rho \rightarrow 0$. This limit could be zero or infinity, or take a finite positive value. Accordingly, the \textit{Hausdorff $f$-measure $\Ha^f$} of $F$ is defined to be
\[
\Ha^f(F)=\lim_{\rho\to 0}\Ha_\rho^f(F).
\]
It is easily verified that Hausdorff measure is monotonic and countably sub-additive, and that $\Ha^s(\emptyset)=0$. Thus it is an outer measure on $\R^n$. In the case that $f(r)=r^s (s>0), $ the measure $\Ha^f$ is the usual $s$-dimensional Hausdorff measure $\Ha^s$. For any subset $F$ one can easily verify that there exists a unique critical value of $s$ at which $\Ha^s(F)$ `jumps' from infinity to zero. The value taken by $s$ at this discontinuity is referred to as the \textit{Hausdorff dimension of  $F$}  and is denoted by $\dim F $; i.e.,
\[
\dim F :=\inf\{s\in \R^+\;:\; \Ha^s(F)=0\}.
\]

When $s$ is an integer, $n$ say, then $\Ha^n$ coincides up to universal constants with standard $n$-dimensional Lebesgue measure. In particular, a set $E \subseteq \R^n$ is null/full with respect to $\Ha^n$ if and only if it is null/full with respect to usual Lebesgue measure. 
Hausorff $s$-measure, like Lebesgue measure, is preserved (up to a constant) by certain well behaved maps.  In particular, if $g: E \to F$ is a bi-Lipshitz bijection between two sets in Euclidean space then $\Ha^s(E) \asymp \Ha^s(F)$.

The proof of Theorem 2 is analogous to the proof of Theorem 3.2 in \cite{BDoKL08} and we will not give it in full detail. However, some subtleties occur, and we give a short outline of the proof with details where these are non-trivial.

For the convergence case the assumption that $\psi$ is monotonic becomes irrelevant. The proof follows by simply using a standard covering argument, which depends on modifying the argument in \S\ref{sec:conv} appropriately for Hausdorff measure. 

In view of the divergence part of Theorems 1a and 1b, proving the divergence part of Theorem 2 is now surprisingly easy due to the remarkable mass transference principle  for linear forms developed by Beresnevich \& Velani~\cite{BV06, BV06_IMRN}. This principle relies upon a `slicing' technique and allows us to transfer statements about the Lebesgue measure of general limsup sets to statements concerning Hausdorff measures. We outline a specialised setup here, tailored to our needs. For consistency, we appeal to the following original notation.

Assume we are given a family $\mc R=(R_\alpha)_{\alpha\in J}$ of lines $R_\alpha \subset \R^2$ indexed by an infinite countable set $J$.   For every $\alpha\in J$ and real $\delta\geq0$ define the $\delta$-neighbourhood of $R_\alpha$ by

\[\Delta(R_\alpha, \delta):=\left\{\x\in \R^2: \: \inf_{\y \in R_\alpha} |\x- \y|<\delta\right\}.\]
Next, assume we are given a non-negative, real-valued function $\Upsilon$ on $J$: \[\Upsilon:J\to\R^+:\alpha\to\Upsilon(\alpha):=\Upsilon_\alpha.\]
Furthermore, to ensure that $\Upsilon_\alpha\to 0$ as $\alpha$ runs through $J$ it is assumed that for every $\epsilon>0$ the set $\{\alpha\in J: \Upsilon_\alpha>\epsilon\}$ is finite. Finally, denote by
\[\Lambda(\Upsilon)=\left\{\x\in\R^2: \x\in \Delta(R_\alpha, \Upsilon_\alpha) \ {\rm for \  infinitely \ many }\: \alpha\in J\right\}\]
the set of points lying in the respective $\Upsilon_\alpha$-neighbourhood of infinitely many of the lines~$R_\alpha$.

\begin{thm}[Beresnevich \& Velani \cite{BV06_IMRN}]\label{mtp}
Let $\mc R$ and $\Upsilon$ be given as above. Let $V$ be a line in $\R^2$ such that
\begin{enumerate}
\item $V\cap R_\alpha\neq\emptyset\quad {\rm for \ all \ } \ \alpha\in J$, and
\item $\sup_{\alpha\in J} \: {\rm diam} (V\cap \Delta(R_\alpha, 1))<\infty.$
\end{enumerate}
Let $f$ and $g:r\to g(r):=r^{-1}f(r)$ be dimension functions such that $r^{-2}f(r)$ is monotonic and let $\Omega$ be a ball in $\R^2$. Suppose for any ball $B$ in $\Omega$ we have
\[\Ha^2\left(B\cap \Lambda\left(g(\Upsilon\right)\right)=\Ha^2(B).\] Then \[\Ha^f\left(B\cap \Lambda\left(\Upsilon\right)\right)=\Ha^f(B).\] 
\end{thm}

We will apply Theorems 1a and 1b to ensure that Theorem \ref{mtp} may be applied. We consider the set $\Lambda(g(\Upsilon))$, where $\Upsilon = \psi(h)/h$  for $R_\alpha = R_{a,b}(c)$ and $h = h_{a,b}$. Up to a universal constant $c > 0$ the set $\Delta(R_\alpha, \Upsilon_\alpha)$ contains one of the strips
	$$
	\ell_{a,b}(c): \, =\, \left\{(x, y) \in [0,1)^2: \, \infabs{a^nx+b^my-c^p} < c \Upsilon_\alpha h_{a,b} \right\}.
	$$
Hence, to prove that the set $\Lambda(g(\Upsilon))$ is full with respect to $\Ha^2$, we need only ensure that the series 
\begin{equation*}
\sum_{(a,b) \in \N^2} \frac{c\Upsilon_\alpha h_{a,b}}{h_{a,b}^{1-1/p}} = c \sum_{(a,b) \in \N^2}  g\left(\frac{\psi(h_{a,b})}{h_{a,b}}\right)h_{a, b}^{1/p}
\end{equation*}
diverges, which is exactly our assumption.

It remains for us to find a line $V$ together with a restricted set of pairs $(a,b)$ satisfying the assumptions of Theorem \ref{mtp}. However, in our proof of Theorems 1a and 1b, we initially introduced the restriction \eqref{eqn:restrictions}, which implies that all lines $R_\alpha$ considered have slope in the interval from $-2$ to $-1/2$. Taking $J = \{(a,b) \in \N^2 : \gcd(a,b) = 1, \frac{1}{2} \le \frac{a^n}{b^m} \le 2\}$ and $V$ to be the line given by the equation $y=x$ clearly gives us properties (1) and (2) from Theorem \ref{mtp}.

We now have all assumptions satisfied and may conclude that for any ball $B$, 
\begin{equation*}
\Ha^f\left(B\cap \Lambda\left(\Upsilon\right)\right)=\Ha^f(B),
\end{equation*}
so that in particular
\begin{equation*}
\Ha^f\left(\Lambda\left(\Upsilon\right)\cap [0,1]^2 \right)=\Ha^f([0,1)^2).
\end{equation*}
It is now a simple matter to verify that $\Lambda\left(\Upsilon\right)\cap [0,1]^2 \subseteq W_{n,m}^p(\psi)$, and Theorem 2 follows. Finally, to deduce Corollaries \ref{cor:HDorder} \& \ref{cor:HD} from Theorem 2, we use standard arguments such as those exhibited in the proofs of Corollaries~3.3 \&~3.4 in \cite{BDoKL08}.

\def\cprime{$'$} \def\cprime{$'$} \def\cprime{$'$} \def\cprime{$'$}
  \def\cprime{$'$} \def\cprime{$'$} \def\cprime{$'$} \def\cprime{$'$}
  \def\cprime{$'$} \def\cprime{$'$} \def\cprime{$'$} \def\cprime{$'$}
  \def\cprime{$'$} \def\cprime{$'$} \def\cprime{$'$} \def\cprime{$'$}
  \def\cprime{$'$} \def\cprime{$'$} \def\cprime{$'$} \def\cprime{$'$}
  \def\cprime{$'$} \def\cprime{$'$}

\end{document}